\documentclass{amsart}
\usepackage{amssymb}

\DeclareMathOperator{\PGL}{PGL}
\DeclareMathOperator{\Rat}{Rat}
\DeclareMathOperator{\Res}{Res}
\DeclareMathOperator{\Fix}{Fix}

\newcommand{\QQ}{\mathbb{Q}}
\newcommand{\CC}{\mathbb{C}}
\newcommand{\PP}{\mathbb{P}}
\newcommand{\ZZ}{\mathbb{Z}}

\newtheorem{theorem}{Theorem}[section]
\newtheorem{lemma}[theorem]{Lemma}
\newtheorem{proposition}[theorem]{Proposition}
\newtheorem{corollary}[theorem]{Corollary}

\newenvironment{definition}[1][Definition]{\begin{trivlist}
\item[\hskip \labelsep {\bfseries #1}]}{\end{trivlist}}
\newenvironment{example}[1][Example]{\begin{trivlist}
\item[\hskip \labelsep {\bfseries #1}]}{\end{trivlist}}

\numberwithin{equation}{section}

\begin{document}

\title{Potential good reduction of degree 2 rational maps}

\author{Diane Yap}
\address{Department of Mathematics, University of Hawaii, 
Honolulu, HI 96822}
\email{dianey@math.hawaii.edu}

\begin{abstract}
We give a complete characterization of degree two rational maps with potential good reduction over local fields. We show this happens exactly when the map corresponds to an integral point in the moduli space $M_2$. We detail an algorithm by which to conjugate any degree two rational map corresponding to an integral point in $M_2$ into a map with unit resultant. The local fields result is used to solve the same problem for fields over a principal ideal domain. Some additional results are given for degree 2 rational maps over $\QQ$.
\end{abstract}

 \maketitle

\section{Introduction}

Motivation for this problem comes from the following result from elliptic curves:
\begin{theorem}\label{eccm} \cite{AEC}\textit{Let $K$ be a local field with a discrete valuation, and let $E/K$ be an elliptic curve. Then $E$ has potential good reduction if and only if its j-invariant is integral}\end{theorem}

The primary goal of this work was to find a criterion in rational maps which gives potential good reduction as having an integral $j$-invariant does for elliptic curves. The $j$-invariant of an elliptic curve $E$ uniquely determines the isomorphism class to which $E$ belongs. In other words, each equivalence class of elliptic curves corresponds to one point in the moduli space of elliptic curves --- the $j$-line.

Because the moduli space of degree $d$ morphisms of the projective line $M_d$ is an affine integral scheme over $\ZZ$  \cite[Remark 4.51]{ADS}, an algebro-geometric argument shows that a degree $d$ rational map has potential good reduction exactly when it corresponds to an integral point in $M_d$. However, this abstract argument does not allow one to actually find the conjugate map with good reduction.

In the current work, we provide a constructive proof of this result in the case $d=2$. The proof relies on specifics about the moduli space $M_2$. For higher degrees, we lack concrete descriptions of the moduli spaces (that is, we lack natural coordinates), so it's not totally clear what one would mean by an ``integral point'' in these cases.

Questions similar to this have garnered much recent attention. Briend and Hsia have reported results like the ones in this article, using similar methods. At the time of writing, these results have not appeared publicly.

Similarly, Bruin and Molnar \cite{molnar} describe an algorithm for finding a minimal model for rational maps of arbitrary degree $d >1$. Their goal for the algorithm is to find maps with many integral points in an orbit. When a rational map defined over $K$ has a model with good reduction also defined over a local field $K$, their algorithm could be used to find this model. Unlike their algorithm, the present work allows for moving to a finite extension of $K$ when necessary. See Section \ref{Examples} for examples of this distinction.

\section{Preliminaries}

In our discourse, we consider rational maps up to conjugation by an element of $\PGL_2$, as this change of coordinates does not affect the geometric properties of the map. To see this conjugation will not change the dynamics of a map $\phi(z)\in K(z)$, observe that for $f\in\PGL_2(K)$,
\[(\phi^f)^n=f^{-1}\circ\phi^n\circ f=(\phi^n)^f\]
The objects we are concerned with are equivalence classes of rational maps which comprise the following quotient space.
 
\begin{definition}
The \emph{moduli space} of rational maps of degree $d$ on $\PP^1$ is the quotient space 
\[M_d = \Rat_d/\PGL_2\]
where $\PGL_2$ acts on $\Rat_d$ via conjugation, $\phi^f = f^{-1}\circ\phi\circ f$.
\end{definition}

This space exists as an affine integral scheme over $\ZZ$, by \cite[Remark 4.51]{ADS}.

We will also use the following notation and definitions from \cite{ADS}:
\begin{center}
\begin{tabular}{lcl}
$K$& \hspace{10mm}& a field with normalized discrete valuation $v: K^* \twoheadrightarrow \ZZ$.\\
$|x|_v$& \hspace{10mm}&$  c^{-v(x)}$ for some $c>1$, an absolute value associated to $v.$\\
$\mathcal{O}_K$& \hspace{10mm}&$ \{\alpha\in K: v(\alpha)\geq 0$\}, the ring of integers of $K$.\\
$\mathfrak{p}$& \hspace{10mm}&$ \{\alpha\in K: v(\alpha)\geq 1$\}, the maximal ideal of $R$.\\
$\mathcal{O}_K^*$& \hspace{10mm}&$ \{\alpha\in K: v(\alpha) = 0$\}, the group of units of $R$.\\
$k$& \hspace{10mm}&$R/{\mathfrak{p}}$, the residue field of $R$.\\
$\sim$& \hspace{10mm}& reduction modulo $\mathfrak{p}$, i.e., $R\to k$, $a\mapsto \tilde{a}$.\\
$\pi$& \hspace{10mm}& uniformizer of  $\mathfrak{p}$.\\
\end{tabular}
\end{center}

\begin{definition}
Let $\phi(z)$ be a rational map over a local field $K$. There are several equivalent conditions by which good reduction is defined, but we will primarily use that $\phi$ has \emph{good reduction}  if $\deg(\phi) = \deg(\tilde\phi)$, where $\tilde{\phi}$ is the reduction of $\phi $ modulo $\mathfrak{p}$. A useful equivalent condition is that $\Res(\phi)$ is a unit modulo $\mathfrak{p}$.  Over a number field $K$, $\phi$ has \emph{good reduction} if it has good reduction at $\mathfrak{p}$ for all prime ideals $\mathfrak{p}$. We say that $\phi$ has \emph{potential good reduction} if there exists an $f\in\PGL_2(\bar{K})$ such that the conjugation $\phi^f$ has good reduction over a finite extension of $K$. \end{definition}

\begin{definition}
To shorten our discourse we define a map $\phi$ as having \emph{genuinely bad reduction} when it does not have potential good reduction. That is, regardless of which $f \in \PGL_2$ we conjugate by, $\phi^f $ has bad reduction. \end{definition}

\begin{definition} The \emph{multiplier} of $\phi$ at a finite fixed point $\alpha$ is the derivative $\lambda_{\alpha}(\phi) = \phi ' (\alpha)$.\end{definition}
\begin{definition}A periodic point $\alpha$ is called
\begin{center}
\begin{tabular}{rl}
\emph{attracting}& if $|\lambda_{\alpha}(\phi)|<1$,\\
\emph{neutral}& if $|\lambda_{\alpha}(\phi)|=1$,\\
\emph{repelling}& if $|\lambda_{\alpha}(\phi)|>1$.\\
\end{tabular}
\end{center}
\end{definition}

We will make frequent use of the following lemma:
\begin{lemma} \cite[Corollary 5.3]{MortonSilverman}\label{julia} Let $\phi: \PP_K^1 \to \PP_K^1$ be a rational map of degree $d\geq 2$ with good reduction at the prime $\mathfrak{p}$.  Let $P\in \PP^1(\bar{K})$ be a periodic point for $\phi$. Then $P$ is $\mathfrak{p}$-adically non-repelling.\end{lemma}

\begin{theorem} \label{ms}  \cite[Theorem 1.14]{ADS}
Let $K$ be an algebraically closed field and let $\phi(z)\in K(z)$ be a rational function of degree $d\geq 2$. Assume that 
\begin{center}$\lambda_P\neq 1$ for all $P\in \Fix(\phi)$.\end{center}
Then
\[\sum_{P\in \Fix(\phi)}\frac{1}{1-\lambda_P(\phi)}=1.\]
\end{theorem}

When $d=2$, we can derive from Theorem \ref{ms} that the symmetric functions $\sigma_1, \sigma_3$ on the multipliers fulfill the following relation:

\begin{equation}\label{3sigmas}\sigma_1 = \sigma_3+2.\end{equation}
Though Theorem \ref{ms} does not hold when one or more of the multipliers is 1, (1) still does. To see this, suppose one of the multipliers equals 1. Then one of the fixed points has multiplicity greater than one, so at least two multipliers are 1. It is easy to see that (1) holds when two or three of the multipliers are 1.

The proof of our main result uses the normal forms from the following lemma:
\begin{lemma} \label{NFL} \cite[Lemma 4.59]{ADS} Let $\phi \in \Rat_2(\CC)$ be a rational map of degree $2$ and let $\lambda_1, \lambda_2, \lambda_3$ be the multipliers of its fixed points. \\
$(a)$ If $\lambda_1 \lambda_2\neq 1$, then there is an $f\in \PGL_2(\CC)$ such that
\[\phi^f(z)=\frac{z^2+\lambda_1 z}{\lambda_2 z + 1}.\]
Further, $\Res(z^2+\lambda_1 z, \lambda_2 z + 1) = 1-\lambda_1\lambda_2$.\\
$(b)$ If $\lambda_1 \lambda_2 = 1$, then $\lambda_1=\lambda_2=1$ and there is an $f\in \PGL_2(\CC)$ such that 
\[\phi^f(z) = z + \sqrt{1-\lambda_3} + \frac{1}{z}.\]
\end{lemma}

\section{Local Results}

In this section, we prove the main result, Theorem \ref{MR}, which characterizes degree two rational maps with potential good reduction over local fields $K$. We address the case $d=2$, as it is known that $M_2 \cong \mathbb{A}^2$ with natural coordinates $(\sigma_1, \sigma_2)$, the first and second symmetric functions on the multipliers of the fixed points of $\phi$\cite{M2}.

Note that $(\sigma_1,\sigma_2)\in M_2$ is integral if and only if all three multipliers, $\lambda_1, \lambda_2, \lambda_3$ are algebraic integers. One direction is clear. For the other direction, suppose that $\sigma_1$ and $\sigma_2$ are integral. Note that $\sigma_3 = \sigma_1-2$ Since $\sigma_1, \sigma_2, \sigma_3$ are the symmetric functions on $\lambda_1, \lambda_2, \lambda_3$, the latter are the roots of the monic polynomial $f(z)=z^3-\sigma_1z^2+\sigma_2z-\sigma_3$. We can conclude that the multipliers are algebraic integers. Since the definition of potential good reduction allows us to go to an extension field, we can then assert that an integral point in $M_2$ is equivalent to integral multipliers.

For the proof, we assume that all three multipliers of a rational map are in $\mathcal{O}_K$, so we begin with a lemma detailing when this does not happen. The main theorem itself is proved via two propositions, one for each normal form in Lemma \ref{NFL}.

\begin{lemma}
\label{L3integral}
If $\lambda_1\lambda_2\neq 1$ and $\lambda_1\lambda_2\equiv 1 \pmod{\pi}$, then $\lambda_3\notin\ \mathcal{O}_K$ if and only if the following conditions all hold:
\begin{align*}
e_1 = e_2 & = e\\
a_1+a_2 &= a\pi^e,\hspace{3mm} a\in\mathcal{O}_K^*\\
a+a_1a_2 & \equiv 0 \pmod{\pi}.
\end{align*}
\end{lemma}

\begin{proof} Suppose $\lambda_1\lambda_2\neq 1$ and $\lambda_1\lambda_2\equiv 1 \pmod{\pi}$. By (\ref{ms}), we can represent $\lambda_3$ in terms of the other two multipliers:

\[\lambda_3 = \frac{2-\lambda_1-\lambda_2}{1-\lambda_1\lambda_2}
= \frac{ a_1\pi^{e_1}+ a_2\pi^{e_2}}{a_1\pi^{e_1}+ a_2\pi^{e_2}+a_1a_2\pi^{e_1+e_2}}.\]

Without loss of generality, suppose that $e_1\leq e_2$, and simplify to obtain
\[\lambda_3=\frac{a_1 + a_2\pi^{e_2-e_1}}{a_1 + a_2\pi^{e_2-e_1}+a_1a_2\pi^{e_2}}.\]

The condition $\lambda_3\notin\mathcal{O}_K$ occurs precisely when a higher power of $\pi$ divides the denominator than the numerator, which happens only when $a_1 + a_2\pi^{e_2-e_1}$ and $a_1a_2\pi^{e_2}$ have the same $\pi$-adic valuation. Since $\pi\nmid a_1$ and $\pi\nmid a_2$, we can simplify the statement to
\[|a_1 + a_2\pi^{e_2-e_1}|_{\pi} = |\pi^{e_2}|_{\pi}.\]

However, since $\pi \nmid a_1$, we can conclude that $e_1=e_2$ and represent $a_1+a_2 = a\pi^e$ with $a \neq 0$ and $\pi\nmid a$. It follows that $e_2=e$. Rewriting with our new information, we get
\[\lambda_3=\frac{ a\pi^e}{\pi^e(a+a_1a_2)}.\]

For the denominator to be divisible by a higher $\pi$ power than the numerator gives us the final condition, that $a+a_1a_2\equiv 0 \pmod{\pi}$.\end{proof}
\begin{theorem}\label{MR}
A degree $2$ rational map $\phi(z)$ over a local field $K$ has potential good reduction if and only if $[\phi]\in M_2(\mathcal{O_K})$.
\end{theorem}

Suppose for $[\phi]\notin M_2(\mathcal{O}_K)$. We remind the reader that an integral point in the moduli space is equivalent to integral multipliers if we allow for a field extension. If $[\phi]$ does not correspond to an integral point in the moduli space, it must have a non-integral multiplier, $\lambda$. Equivalently, $|\lambda|_v>1$ so $\phi(z)$ has a repelling fixed point. By Lemma \ref{julia}, $\phi(z)$ has genuinely bad reduction. 

The other direction will be proved using two propositions --- one for each of the two forms in (\ref{NFL}). 
\begin{proposition}\label{form1}
Suppose that the multipliers $\lambda_1, \lambda_2, \lambda_3$ of $\phi(z)$ are all integral.
Let  
\[\phi(z)= z+\sqrt{1-\lambda_3}+\frac{1}{z}\]
with $\lambda_1\lambda_2=1$. Then $\phi(z)$ has good reduction. \end{proposition}

\begin{proof}
Given $\lambda_3\in \mathcal{O}_K$, it follows that $\sqrt{1-\lambda_3}$ is also integral. Since the coefficients of $\phi(z)$ are all integral, we can calculate the resultant 
\[\Res(\phi^f(z)) = \Res(z^2+\sqrt{1-\lambda_3}z+1, z) = 1.\]
We can therefore conclude that in this case $\phi(z)$ has  good reduction. 
\end{proof}

\begin{proposition}\label{form2}
Suppose that the multipliers $\lambda_1, \lambda_2, \lambda_3$ of $\phi(z)$ are all integral.
Let \[\phi(z)=\frac{z^2+\lambda_1 z}{\lambda_2 z + 1}.\] If $\lambda_1\lambda_2\neq 1$, then $\phi(z)$ has potential good reduction.
\end{proposition}

\begin{proof}
Recall that by Lemma \ref{NFL},

\[\Res(z^2+\lambda_1 z,\lambda_2 z + 1,z)=1-\lambda_1\lambda_2.\]

If $\lambda_1\lambda_2\not\equiv 1 \pmod{\pi}$, then $\Res(\phi(z))\neq 0 \pmod{\pi}$, so $\phi(z)$ has good reduction. Now suppose $\lambda_1\lambda_2\equiv 1 \pmod{\pi}$. Here, we can show by equation (\ref{3sigmas}) that  $\lambda_1\equiv 1 \pmod{\pi}$ and $ \lambda_2\equiv 1 \pmod{\pi}$:
\begin{align*}
\sigma_1 &= \sigma_3 + 2\\
\lambda_1 + \lambda_2 + \lambda_3 &= \lambda_1 \lambda_2 \lambda_3 + 2\\
\lambda_1 + \lambda_2 &\equiv   2 \pmod{\pi}.\\
\end{align*}
Now, we may substitute and use our assumption here that $\lambda_1 \lambda_2\equiv 1 \pmod{\pi}$ to get the following:
\begin{align*}
\lambda_1(2-\lambda_1)&\equiv   1 \pmod{\pi}\\
(\lambda_1-1)^2 &\equiv   0 \pmod{\pi}\\
\lambda_1 &\equiv   1 \pmod{\pi}.
\end{align*}
It follows that $\lambda_2\equiv 1 \pmod{\pi}$ must hold too. We can represent $\lambda_1$ and $\lambda_2$ as follows, with $a_1,a_2\in\mathcal{O}_K$, $e_1,e_2 >0$, $\pi\nmid a_1$, and $\pi\nmid a_2$:

\begin{align*}
\lambda_1&= 1 + a_1\pi^{e_1}\\
\lambda_2&= 1 + a_2\pi^{e_2}.
\end{align*}

For $\lambda_3\in\mathcal{O}_K$ to hold, at least one of the conditions in Lemma \ref{L3integral}must fail.

\textbf{Case 1}\\
Suppose $e_1\neq e_2$. Then, without loss of generality, we may assume $e_1<e_2$. 
\begin{align*}
\Res(\phi) &= 1-\lambda_1\lambda_2\\
&= -(a_1\pi^{e_1}+a_2\pi^{e_2}+a_1a_2\pi^{e_1+e_2})\\
&= -\pi^{e_1}(a_1+a_2\pi^{e_2-e_1}+a_1a_2\pi^{e_2}).
\end{align*}

By assumption $\pi\nmid a_1$, so the order of $\pi$ in $\Res(\phi)$ is $e_1$. To see that $\phi(z)$ has potential good reduction, first conjugate by $f(z)=z-1$ to obtain

\[\phi^f(z) = \frac{z^2 + (\lambda_1+\lambda_2-2)z +2 - \lambda_1 - \lambda_2}{\lambda_2z-\lambda_2+1}.\]

Now conjugate again by $g(z) =cz$ (with $c=\sqrt{ \pi^{e_1}}$) to get

\begin{equation}(\phi^f)^g(z)= \frac{c^2z^2+(\lambda_1+\lambda_2-2)cz + 2-\lambda_1-\lambda_2}{c^2\lambda_2z+c(1-\lambda_2)}.\end{equation}

Since $\lambda_1+\lambda_2-2 = a_1\pi^{e_1}+a_2\pi^{e_2} =\pi^{e_1}(a_1+a_2\pi^{e_2-e_1})$, with $\pi \nmid a_1$, we can write $c^2m=\lambda_1+\lambda_2-2$ with $\pi\nmid m$. Similarly, we can let $c^2n = 1-\lambda_2$. With those substitutions, $(3)$ may be rewritten as

\[(\phi^f)^g(z) = \frac{z^2 + cmz -m}{\lambda_2z+cn}.\]
with all coefficients in $\mathcal{O}_L$, where $L=K(c)$, a finite extension of $L$. We can calculate the resultant 

\begin{align*}
\Res((\phi^f)^g(z))&=-\lambda_2^2m-\lambda_2c^2mn+c^2n^2\\
&=c^2n^2+c^2mn-m.
\end{align*}

Since the resultant is in $\mathcal{O}_K$, it's enough to verify that $\pi\nmid\Res((\phi^f)^g(z))$. Recall that $\pi\nmid m$, so we know that $\Res((\phi^f)^g(z)) \not\equiv 0 \pmod{\pi}$, so $\phi(z)$ has potential good reduction.

Using the above defined substitutions, $\lambda_1= c^2(m+n)+1$ and $\lambda_2=1-c^2n$. In particular,

\[\Res(\phi) = 1-\lambda_1\lambda_2 = c^2(c^2n^2+c^2mn-m).\]
Note that 
\begin{equation}\label{res}\Res(\phi) = c^2\Res((\phi^f)^g(z)).\end{equation}

\textbf{Case 2}\\
Suppose that $e_1=e_2=e$, but $a_1+a_2\neq a\pi^e$ for any $a\in \mathcal{O}_K$. We may write $a_1+a_2=a\pi^d$ with $d<e$ and $\pi\nmid a$. 

\begin{align*}
\Res(\phi) &= 1-\lambda_1\lambda_2\\
&= -\pi^{e}(a_1+a_2+a_1a_2\pi^{e})\\
&=-\pi^{e+d}(a+ a_1a_2\pi^{e-d}).
\end{align*}

Now let $f(z)=z-1$ and $g(z) = cz$ with $c=\sqrt{\pi^{e+d}}$. Conjugating by $f$ then $g$ gives us equation (3). Now write $\lambda_1+\lambda_2-2=\pi^e(a_1+a_2)=a\pi^{e+d}=ac^2$ and $1-\lambda_2=a_2\pi^e=cn$, then substitute to get
\[(\phi^f)^g(z) = \frac{z^2 + acz -a}{\lambda_2z+n}.\]
The resultant is 
\begin{align*}
\Res((\phi^f)^g) &= -\lambda_2^2-a\lambda_2cn+n^2\\
&=-a+acn+n^2.
\end{align*} 
Since $\pi$ divides both the 2nd and 3rd terms, but not $a$, $\Res((\phi^f)^g)\not\equiv 0 \pmod{\pi}$. Thus $\phi(z)$ has potential good reduction.

Using the above defined substitutions, $\lambda_1= ac^2+1+cn$ and $\lambda_2=1-cn$. In particular,

\[\Res(\phi) = 1-\lambda_1\lambda_2 = k^2(-a+acn+n^2).\]
Note that again, $\Res(\phi) = c^2\Res((\phi^f)^g(z))$.

\textbf{Case 3}\\
Suppose $e_1=e_2=e$ and $a_1+a_2= a\pi^e$ but $a+a_1a_2\not\equiv 0 \pmod{\pi}$.
\begin{align*}
\Res(\phi) &= 1-\lambda_1\lambda_2\\
&= -(a_1\pi^{e_1}+a_2\pi^{e_2}+a_1a_2\pi^{e_1+e_2})\\
&= -\pi^{e}(a_1+a_2+a_1a_2\pi^{2e})\\
&=-\pi^{2e}(a+a_1a_2).
\end{align*}

By our assumption that $a+a_1a_2\not\equiv 0 \pmod{\pi}$, we have that $\pi$ has order $2e$ in $\Res(\phi)$. As before, conjugate $\phi(z)$ first by $f(z)=z-1$ then by $g(z) = cz$, this time with $c=\pi^e$ to get an equation identical to (3). Now note that $\lambda_1+\lambda_2-2 = (a_1+a_2)\pi^e=a\pi^{2e}=ac^2$, and let $1-\lambda_2=-a_2\pi^e=-a_2c$. With these substitutions, we have

\begin{equation}(\phi^f)^g(z)= \frac{z^2+acz -a}{\lambda_2z-a_2}.\end{equation}

Here the resultant is $\Res((\phi^f)^g)= a_2^2+aa_2\lambda_2c-a\lambda_2^2$.
\begin{align*}
\Res((\phi^f)^g)\equiv 0 \pmod{\pi}& \iff a_2^2-a\lambda_2^2+aa_2\lambda_2c\equiv 0 \pmod{\pi}\\
&\iff a_2^2-a\lambda_2^2 \equiv 0 \pmod{\pi}\\
&\iff a_2(a\pi^e-a_1)-a(1+2a_2\pi^e+a^2\pi^{2e})\equiv 0 \pmod{\pi}\\
&\iff a_1a_2+a\equiv 0 \pmod{\pi}.
\end{align*}

Since we assumed $a_1a_2+a\not\equiv 0 \pmod{\pi}$, it follows that $\Res((\phi^f)^g)\not\equiv 0 \pmod{\pi}$, so $\phi(z)$ has potential good reduction.

Using the above defined substitutions, $\lambda_1= ac^2+1-a_2c$ and $\lambda_2=1-a_2c$. In particular we may rewrite
\[\Res((\phi^f)^g)= a_2^2+-a-aa_2c\]
\[\Res(\phi) = 1-\lambda_1\lambda_2 = c^2(a_2^2-a-aa_2c).\]

Note that once again, $\Res(\phi) = c^2\Res((\phi^f)^g(z))$.
\end{proof}

\section{Global Results}
\subsection{Quadratic Polynomials}

\begin{lemma}\label{quadratic}  \cite[Statement 4.12]{ADS} For $K$ of characteristic not equal to $2$, every quadratic polynomial $f(z)=Az^2+Bz+C \in K[z]$ is linearly conjugate over $K$ to a unique polynomial of the form $\phi(z)=z^2+c$ . \end{lemma}

\begin{theorem} Let $\phi(z) \in K[z]$ be a quadratic polynomial over a number field. Then $\phi$ has potential good reduction if and only if $[\phi]$ is an integral point in the moduli space $M_2$. \end{theorem}

\begin{proof}
By Lemma \ref{quadratic}, we may assume $\phi(z)=z^2+c$.

The function $\phi(z)$ has fixed points at $q_{\pm} = \frac{1\pm \sqrt{1-4c}}{2}$, with corresponding multipliers $\lambda_{q_{\pm}}=1\pm \sqrt{1-4c}$.  

Suppose that $\phi(z)$ has potential good reduction. Then by Theorem \ref{MR}, all multipliers are $\mathfrak{p}$-adically integral for every prime $\mathfrak{p}$. We can calculate:

\begin{align*}
|\lambda_{q_+}|_{\mathfrak{p}} \leq1&\Leftrightarrow |1-4c|^{1/2}_{\mathfrak{p}} \leq 1\\
&\Leftrightarrow |1-4c|_{\mathfrak{p}} \leq 1\\
&\Leftrightarrow |4c|_{\mathfrak{p}} \leq 1.\\
\end{align*}

So potential good reduction means $|4c|_p\leq 1$. Or, in other words, that $4c$ is a $p$-adic integer for that $\mathfrak{p}$.  The symmetric functions on the multipliers of the fixed points of $\phi(z)$ are $\sigma_1 = 2$ and $\sigma_2 = 4c$, so $\phi(z)$ corresponds to the point $(2, 4c)$, which is integral the moduli space.

Now suppose that  $[\phi(z)] = [z^2+c]$ is integral in the moduli space. In particular, it corresponds to the point $(2,4c)$ which is integral when $|4c|_p \leq 1$. The following steps show how to find an $f(z) \in \PGL_2$ to conjugate by to obtain a rational map with good reduction. First, we find the fixed points of $\phi(z)$, which are $z = \frac{1\pm \sqrt{1-c}}{2}$. Then let $f(z) = z + \frac{1+ \sqrt{1-c}}{2}$. Conjugating gives $\phi^f(z) = z^2 + (1+ \sqrt{1-c})z$, which has good reduction over the quadratic extension field $K[\sqrt{1-c}]$. Therefore $\phi(z)$ has potential good reduction.  \end{proof}

We have some additional results when $K=\QQ$:

\begin{lemma}
Let $\phi(z) = z^2 + \frac{k}{4}$ with $k\in\ZZ$. Then if $k\equiv 0,1 \pmod{4}$, there exist $B, C\in\ZZ$ such that $\phi(z)$ is conjugate to $z^2+Bz+C$, which has good reduction.
\end{lemma}

\begin{proof}
Let $f(z)=z+\frac{B}{2}$, (where $B$ is as yet to be determined). Then conjugation gives
\[f^{-1}\circ \phi\circ f(z) = z^2 + Bz +\left(\frac{B^2}{4}-\frac{B}{2}+\frac{k}{4}\right).\]
For whatever $B\in\ZZ$ we choose, it must hold that $C=\frac{B^2-2B+k}{4}\in\ZZ$. That is, for some $x\in\ZZ$, $B^2-2B+k = 4x$. Solving for $B$ yields
\[B = 1\pm\sqrt{1- k+4x}\in\ZZ.\]
So there is some $y\in\ZZ$ for which $\sqrt{1- k - 4x}=y$. Squaring both sides and rearranging, we get that $y^2 \equiv 1- k \pmod{4}$. The only quadratic residues in $\ZZ/4\ZZ$ are $0$ and $1$, so it must hold that $k\equiv 0,1 \pmod{4}$.

The case $k\equiv 0 \pmod{4}$ is exactly the case where $\phi(z)=z^2+C$ with $C = \frac{k}{4}\in\ZZ$, so no $B$ needs to be chosen, as $\phi(z)$ already has good reduction. In the other case, $k \equiv 1 \pmod{4}$, we set $B=1$, and $C = \frac{k-1}{4}$.
\end{proof}

\begin{proposition}
Let $\phi(z) = z^2  +  \frac{k}{4}$ with $k\in\ZZ$. Then $\phi(z)$ is linearly conjugate to a morphism $g(z)\in\QQ(z)$ with good reduction over $\QQ$ if only if $k \equiv 0, 1 \pmod{4}$.
\end{proposition}
\begin{proof}
The forward direction is a result of the previous lemma. For the other direction, suppose  $\phi(z) = z^2 + \frac{k}{4}$ and $f(z) = \frac{az+b}{cz+d}\in \PGL_2$. Then the full conjugation, $g(z) = \phi^f(z)$ looks like:
 \[ \frac{(4a^2d + c^2dk - 4bc^2)z^2 + (8abd + 2cd^2k - 8bcd)z + 4b^2d + d^3k - 4bd^2}{(4ac^2 - 4a^2c - c^3k)z^2 + (8acd - 8abc -  2c^2dk)z + 4ad^2 - 4b^2c - cd^2k }.\]  

First, note that $c$ and $d$ cannot both be $0$, since $f(z) \in \PGL_2$. Now consider the case $c = 0$. Substituting and reducing in the above equation gives
\[g(z) =  \frac{4a^2z^2 + 8abz + 4b^2 + d^2k - 4bd}{4ad}.\]

For $g(z)$ to have good reduction, it must hold that $d|a$, $d|2b$ and $2|d$. So we can change notation and replace $a$ with $ad$ and $2b$ with $bd$. Then we can rewrite $g(z)$ as
\[ g(z) = \frac{4a^2z^2 + 4abz + b^2 - 2b + k}{4a}.\]

Now we need $\frac{b^2 - 2b + k}{4a}  = x $ for some $ x \in \ZZ$. Solving for $b$, we get 

\[b = 1 \pm \sqrt{1 - k + 4ax}.\]

So there is some $y \in \ZZ$ such that $y^2 = 1- k + 4ax$. In other words, $y^2 \equiv 1 - k \pmod{4}$. Since the only quadratic residues in $\ZZ/4\ZZ$ are $0$ and $1$, we have that $k \equiv  0, 1 \pmod{4}$ when $c = 0$. 

Next, consider the case $d = 0$. Then we have 
\[g(z) = \frac{4bcz^2}{(4a^2 - 4ac + c^2k)z^2 + 8abz + 4b^2 }.\]
For $g(z)$ to have good reduction the coefficient $-4bc$ in the numerator must cancel. So we have  these divisibility properties: $c|2a$ and $c|b$. As before, replace $2a$ with $ac$ and $b$ wtih $bc$. Then we can rewrite $g(z)$ as
\[g(z) = \frac{4bz^2}{(a^2 - 2a + k)z^2 + 4abz + 4b^2}. \]

Now, we just need to ensure that $\frac{a^2 - 2a + k}{4b} = x$ for some $x \in \ZZ$. Solving for $a$, we get 

\[a = 1 \pm \sqrt{1 - k + 4bx}. \]

There must be some $y \in \ZZ$ such that $y^2 = 1 - k - 4bx$, i.e. $y^2 \equiv 1-k \pmod{4}$. Thus we have that $k \equiv 0,1 \pmod{4}$. 

Now consider the case where both $c$ and $d$ are nonzero. Without cancellation, $g(z)$ reduces to the constant $-\frac{d}{c}$ in $\ZZ/2\ZZ$, indicating bad reduction. The only possibility for good reduction requires that each monomial be divisible by $2$, requiring that $c$ and $d$ both be even. 

Knowing that $c$ and $d$ are both even, we may rewrite $f(z)$ as $f(z) =  \frac{az+b}{2cz+2d}$. Then $g(z)$ once again reduces to  $-\frac{d}{c}$ in $\ZZ/2\ZZ$ unless there is cancellation. Suppose that $2^n | c$ and $2^n | d$. Then we can repeat the process, continuing to rewrite $f(z)$ as above and recalculating  $g(z)$ , indicating that $g(z)$ reduces to  $-\frac{d}{c}$ for arbitrarily large $n$. Since there are no $c$ and $d$ in  $\ZZ$ which would permit cancellation, $\phi(z) = z^2 + \frac{k}{4}$ cannot be conjugate to a function with good reduction over $\QQ$ unless $k \equiv 0, 1 \pmod{4}$, as in the above cases.
\end{proof}
\subsection{Quadratic Rational Maps}

Our result over principal ideal domains is a corollary to Theorem \ref{MR} and the proof follows an identical format, so we will be brief.

\begin{corollary}
Let $K$ be a field such that $\mathcal{O}_K$ is a principal ideal domain. Then a degree $2$ rational map $\phi(z)\in K(z)$ has potential good reduction if and only if $[\phi]\in M_2(\mathcal{O}_K)$.
\end{corollary}

\begin{proof}Recall from the proof of the local fields case that the change of variables we employed gave us a resultant which was no longer divisible by the prime $\pi$ of bad reduction, or more precisely, (\ref{res}). Note that Proposition \ref{form1} holds as a global result. 

We remind the reader that $\mathfrak{p}_i=(\pi_i)$ since $\mathcal{O}_K$ is a principal ideal domain. Without loss of generality, let $\phi(z)$ be in the normal form of Lemma \ref{NFL}. To extend Proposition \ref{form2} to $\mathcal{O}_K$, we consider the prime factorization of the ideal 

 \[\Res(\phi)\mathcal{O}_K = \prod_{i=1}^n\mathfrak{p}_i^{e_i}.\]

 For each of the $\mathfrak{p}_i$ there are two cases: either $\mathfrak{p}_i$ has the property that $\lambda_1\lambda_2-1\in \mathfrak{p}_i$ or it doesn't. In the case that it does not, the proof of Proposition \ref{form2} shows that $\phi(z)$ has good reduction for that $\mathfrak{p}_i$, so we need only be concerned with $\mathfrak{p}_i$ for which $\lambda_1\lambda_2-1\in \mathfrak{p}_i$.

To that end, let $\mathfrak{p}_i$ for $i = 1, 2, \ldots, m$ be the complete list of prime ideals such that $\lambda_1\lambda_2-1\in \mathfrak{p}_i$ and $\Res(\phi) \in\mathfrak{p}_i$. Following the proof of Proposition \ref{form2}, let $e_i$ be the order of $\mathfrak{p}_i$ in $\Res(\phi)\mathcal{O}_K$, and set

\[k = \sqrt{\prod_{i=1}^m \pi_i^{e_i}}.\]

By (\ref{res}), 
\[\Res(\phi(z))\mathcal{O}_K= \Res(\phi^f(z))\prod_{i=1}^m \mathfrak{p}_i^{e_i} .\] 

In particular, the ideal generated by the resultant of the conjugated map is no longer divisible by any $\mathfrak{p}$ with the property $\lambda_1\lambda_2-1\in \mathfrak{p}$. We may conclude that $\phi(z)$ has potential good reduction.\end{proof}

\section{Examples}\label{Examples}

Here, we give several examples of finding the minimal resultant of a given degree $2$ rational map, first using methods from this article, then also by using the algorithm from \cite{molnar}. The following lemma is referenced throughout:

\begin{lemma}\label{minimal}\cite[Lemma 3.1]{molnar}
If $d$ is even and $v(\Res_d(F,G))<d$ or if $d$ is odd and $v(\Res_d(F,G))<2d$ then $[F,G]$ is an $R$-minimal model for $[\phi]$.

\end{lemma}

\begin{example} Let $\phi(z)=\frac{z^2-2z}{-2z+1}$. $\Res(\phi)=-3$, and we verify that $\lambda_1\lambda_2=(-2)(-2)\equiv 1 \pmod{3}$. Now, conjugate first by $f(z)=z-1$, then by $g(z)=\sqrt{3}$ to get \[(\phi^f)^g=\frac{z^2-2\sqrt{3}z+2}{-2z+\sqrt{3}}\]
which has resultant $1$. 

Since we are working with the $3$-adic valuation, $v(\Res_d(F,G)) =  1<2=d$. By Lemma \ref{minimal}, $\phi(z)$ is already locally minimal, and the algorithm of Bruin and Molnar would not change it.
\end{example}

\begin{example} Let $\phi(z)=\frac{z^2+9z}{z+1}$. $\Res(\phi)=-8$, and we verify that $\lambda_1\lambda_2=(9)(1)\equiv 1 \pmod{2}$. Now, conjugate first by $f(z)=z-1$, then by $g(z)=2\sqrt{2}$ to get \[(\phi^f)^g=\frac{z^2+2\sqrt{2}z-1}{z}\]
which has resultant $-1$. 

Note that both the numerator and the denominator of $\phi(z)$ are monic and that the degree of the denominator is half the degree of the numerator. By \cite[Remark 3.4]{molnar}, $\phi(z)$ is minimal according to the algorithm of Bruin and Molnar.
\end{example}

\begin{example} Let $\phi(z)=\frac{z^2-3z}{-z+1}$. $\Res(\phi)=-2$, and we verify that $\lambda_1\lambda_2=(-3)(-1)\equiv 1 \pmod{2}$. Now, conjugate first by $f(z)=z-1$, then by $g(z)=\sqrt{2}$ to get \[(\phi^f)^g=\frac{z^2-3\sqrt{2}z+3}{-z+\sqrt{2}}\]
which has resultant $1$. 

We are working with the $2$-adic valuation, so $v(\Res_d(F,G)) =  1<2=d$. By Lemma \ref{minimal}, $\phi(z)$ is found by the algorithm of \cite{molnar} to be locally minimal.

\end{example}

Finally, consider the following example which has a resultant divisible by more than one prime $p$:

\begin{example} Let $\phi(z)=\frac{z^2+13z}{z+1}$. $\Res(\phi)=-12$, and we verify that $\lambda_1\lambda_2=(1)(13)\equiv 1 \pmod{2}$ and also $\lambda_1\lambda_2\equiv 1\pmod{3}.$ Now, conjugate first by $f(z)=z-1$, then by $g(z)=\sqrt{12}$ to get 
\[(\phi^f)^g=\frac{z^2+2\sqrt{3}z-1}{z}\]
which has resultant -1. 

Note that both the numerator and the denominator of $\phi(z)$ are monic and that the degree of the denominator is half the degree of the numerator. Again by \cite[Remark 3.4]{molnar}, $\phi(z)$ is considered to be minimal.

\end{example}

\section{Acknowledgements}

I am grateful for my advisor, Michelle Manes. Her endless patience and invaluable guidance are unparalleled.


\begin{thebibliography}{9}
 
 \bibitem{molnar} N. Bruin, A. Molnar. \emph{Minimal models for rational functions in a dynamical setting}, {\tt arXiv:1204.4967 [math.NT]}.
 
 \bibitem{MortonSilverman} P. Morton, J.H. Silverman, \emph{Periodic points, multipliers, and dynamical units}, J. reine angew. Math. \textbf{461}, 81--122.
\bibitem{ADS} J.H. Silverman, \emph{The Arithmetic of Dynamical Systems}, Springer-Verlag, New York, NY, 2007.

\bibitem{AEC} J.H. Silverman, \emph{The Arithmetic of Elliptic Curves}, Springer-Verlag, New York, NY, 1986.

\bibitem{M2} J.H. Silverman, \emph{The space of rational maps on $\PP^1$}, Compos. Math. 98 (1995), 269-304.

\end{thebibliography}
\end{document}